\documentclass[a4paper,11pt]{article} 
\usepackage{amsmath,amssymb}
\usepackage{graphicx,url}

\newcommand{\RR}{\mathbb{R}}
\newcommand{\NN}{\mathbb{N}}

\newcommand{\ZZ}{\mathbb{Z}}

\newtheorem{Theorem}{Theorem}
\newtheorem{Proposition}{Proposition}
\newtheorem{Definition}{Definition}

\newtheorem{Example}{Example}

\title{Self-similar Delone sets and Pisot numbers}
\author{Christoph Bandt and Yves Meyer}
\begin{document}

\maketitle

\begin{abstract}
We consider Delone point patterns with self-similarity. Under mild conditions, the similarity factor is a Pisot number if and only if the pattern is uniformly discrete. The classical case is a Meyer set $\Lambda$ with $\Lambda\supset \theta\Lambda$ for some $\theta>1,$ for which $\theta$ must be a Pisot number or a Salem number. When $\Lambda$ contains several similar copies of itself, the case of a Salem number drops out for $\theta<2.$ On the other hand, strictly self-similar patterns with a Pisot factor must be Meyer sets. Various examples are given. 
\end{abstract}

\section{When uniformly discrete patterns have a Pisot factor} \label{s1}
This paper studies the interplay between geometric and algebraic properties of discrete point sets $\Lambda\subset \mathbb{R}^n.$  We start with some basic definitions.  Let 
\begin{equation} h(\Lambda)=\inf\{|x-y|, \,x\in \Lambda, y\in \Lambda, x\neq y\} \ . \label{hh}
\end{equation}
We say that $\Lambda$ is uniformly discrete if $h(\Lambda)>0.$  A Delone set is a uniformly discrete set $\Lambda$ which is relatively dense: there is a positive $R$ such that for every $x\in \mathbb{R}^n$ one can find a $\lambda\in \Lambda$ with $|x-\lambda|\leq R.$   A Delone set $\Lambda\subset \mathbb{R}^n$ is called a Meyer set if there exists a finite set $E\subset \mathbb{R}^n$ such that $\Lambda-\Lambda \subset \Lambda+E.$ By $\Lambda-\Lambda$ we denote the set of all differences $x-y, x,y \in \Lambda.$  Lagarias proved that a Delone set $\Lambda$ is a Meyer set if and only if  $\Lambda-\Lambda$ is also a Delone set \cite{Lag96,Mo97}. 

Lattices like $\ZZ^n,$ for which $\Lambda-\Lambda=\Lambda,$ are the simplest examples of Meyer sets. We are interested in aperiodic examples which were used in quasicrystal modelling. Their construction usually involves some kind of self-similarity.  Throughout we consider a real similarity factor $\theta >1.$ The simplest form of self-similarity of $\Lambda$ is given by the relation  
\begin{equation} \Lambda\supset \theta\Lambda \ . \label{e2}\end{equation}
For a Meyer set, this condition implies that $\theta$ is either a Pisot or a Salem number \cite[Chapter I]{Meyer72}. These are algebraic integers, that is, roots of polynomials with integer coefficients and leading coefficient one. $\theta$ is a Pisot number if all other roots $\theta'$ have modulus smaller than one. For a Salem number the roots $\theta'$ have modulus less or equal to one, and at least one root should have modulus 1, so that the two types of numbers are distinct. 

In this essay we are looking for the Delone sets $\Lambda$ which satisfy 
\begin{equation} \Lambda = \theta\Lambda +F \ \mbox{ where $F$ is a finite set.}
\label{selF}\end{equation}

{\bf Conjecture } \emph{If a Delone set $\Lambda$ fulfils \eqref{selF} then $\theta$ must be a Pisot number, and $\Lambda$ must be a Meyer set.} \medskip

The second conjecture is a consequence of the first, as shown by Theorem \ref{t3} below under natural assumptions. It seems interesting that one inclusion of  \eqref{selF} follows from \eqref{e2}.

\begin{Proposition} If a Meyer set fulfils $\Lambda\supset\theta\Lambda$ then there is a finite set $F$ with $ \Lambda \subseteq \theta\Lambda +F .$
\end{Proposition}

\noindent {\it Proof. } Let $\lambda\in\Lambda .$ Since $\theta\Lambda$ is relatively dense there exists a $\lambda'\in\Lambda$ such that $|\lambda -\theta\lambda'|\le R.$ But $y=\theta\lambda'\in\Lambda$ which implies that $\lambda-y\in F$ where $F$ is the intersection between $\Lambda-\Lambda$  and the ball centered
at 0 with radius $R.$ Since $\Lambda-\Lambda$ is Delone, $F$ is finite.
\hfill $\Box$ \smallskip 

We now prove the conjecture for $\theta<2.$ 
 
\begin{Theorem}\label{t1}
If $\Lambda$ is a uniformly discrete solution of \eqref{selF} and $\theta<2,$  then $\theta$ is a Pisot number. 
\end{Theorem} 

We prove a more general statement for sets $F$ which contain $m+1$ regularly spaced points on a line. Theorem \ref{t1} concerns the case $m=1.$  Note that for $\Lambda\not= \{ 0\}$ the set $F$ in \eqref{selF} must have at least two points. We can assume $0\in F$ by choice of the coordinate origin.  In other words, for $x_0\in F$ the translate $\Lambda'=\Lambda-x_0$ and $F'=F+(\theta-1)x_0$ fulfil \eqref{selF}.  

\begin{Theorem}\label{t1a}
If $\Lambda$ is uniformly discrete and  $\Lambda \supseteq \theta\Lambda + \{ 0,1,...,m\}\cdot f$ for some $f\in \RR^n \setminus\{ 0\}$ and an integer $m>\theta -1$ then $\theta$ is a Pisot number. 
\end{Theorem}

Here $\Lambda$ need not be a Delone or Meyer set. The assertion is a consequence of an important result of Feng \cite{Feng16} who considered the following sets on the real line.
\begin{eqnarray}S(\theta, m)&=&\left\{ s=\sum_0^K \epsilon_k\theta^k\, | \,  K\in\NN , \epsilon_k\in \{0, 1, \ldots, m\} \right\} \ \mbox{ and }\label{Sm}\\
R(\theta, m)=S(\theta, m)-S(\theta, m)&=&\left\{ r=\sum_0^K \epsilon_k\theta^k\, | \,  K\in\NN , \epsilon_k\in \{-m, \ldots, m\} \right\}\, .
\label{Rm}\end{eqnarray}

Note that $S(\theta,m)$ is a solution of \eqref{selF} but not a Delone set since it consists of positive numbers. The difference set $R(\theta,m)$ is a solution of \eqref{selF} with a greater $F,$ and is a Delone set for $m>\theta-1.$ Now we reformulate Theorem 1.2 and Theorem 1.4 of \cite{Feng16} for our purposes. \medskip

\noindent{\bf Feng's Theorem. } \emph{If $S(\theta, m)$ is uniformly discrete (or, equivalently, 0 is isolated in $R(\theta, m)$) and $m>\theta -1$ then $\theta$ is a Pisot number.} \vspace{2ex}

\noindent {\it Proof of Theorem \ref{t1a}. }
We substitute the right-hand side of the assumption into $\theta\Lambda$ to obtain
\[ \Lambda \supseteq \theta [\theta\Lambda + \{ 0,1,...,m\}\cdot f]+ \{ 0,1,...,m\}\cdot f =
\theta^2\Lambda+\{\, \sum_{j=0}^1 \epsilon_j\theta^j\, | \, \epsilon_j\in \{0, 1, \ldots, m\} \,\}\cdot f \ .  \]
Repeating this substitution we get for $k=2,3,...$  
\[\Lambda \supseteq \theta^k \Lambda  + M_k \quad\mbox{ where } M_k=\{\, \sum_{j=0}^k \epsilon_j\theta^j\, | \, \epsilon_j\in \{0, 1, \ldots, m\} \,\}\cdot f \ .\]
We fix a point $c\in\Lambda .$ Then $M_k\subseteq \Lambda-\theta^kc$ implies $h(M_k)\ge h(\Lambda)>0$ for all $k.$ Moreover, the $M_k$ form an increasing sequence. Thus
$h(\bigcup M_k)\ge \inf h(M_k) \ge h(\Lambda) >0  .$ 
Since $\bigcup M_k = S(\theta, m)\cdot f ,$ we can now apply Feng's theorem.
\hfill $\Box$ \medskip

\section{A related proof by classical arguments} \label{s2}
Feng's proof is by no means self-contained. As noted in \cite[Section 1]{Feng16}, it is based on the work of several predecessors including Erd\"os, Jo\'o, Komornik, and Akiyama.  For that reason, we now give a more direct proof of a weaker statement based on classical arguments and the following definition.

\begin{Definition} 
A set $S$ of real numbers is a coherent set of frequencies if there exist a compact interval $J$ and a constant $C$ such that, for every trigonometric sum $P=\sum_{s\in S}a(s)\exp(2\pi i sx),$ we have 
\begin{equation}
\sup_{x\in \mathbb{R}}|P(x)|\leq C\sup_{x\in J}|P(x)|. 
\label{e3}\end{equation}
\end{Definition} 

Equivalently, $S$ is not a coherent set of frequencies if there exists a sequence $J_k$ of intervals whose lengths $l_k$ tend to $\infty ,$ and a sequence $P_k$ of trigonometric polynomials whose frequencies belong to $S,$ such that $\sup_{J_k}|P_k|\to 0$ while $\|P_k\|_{\infty}=1.$
\medskip

A famous theorem of Pisot implies that for any $\theta >1$ which is NOT Pisot, the series $\sum_k \sin^2(\theta^k\pi x)$ diverges whenever $x\neq 0,$ cf.~\cite[Theorem XII, Section I.6.4]{Meyer72}. Salem noted in 1944 the consequence for non-Pisot numbers $\theta$ that the products
\[ Q_k(x)=\cos(2\pi x)\cos(2\pi \theta x)\cdots\cos(2\pi \theta^k x) \]
converge to $0$ uniformly on any compact interval not containing $0.$ Note that the formula $2\cos a_0 \cos a_1 =  \cos (a_0-a_1)+\cos (a_0+a_1)$ can be extended to finite products by induction: $2^k\prod_{j=0}^k\cos a_j =\{ \sum \cos(a_0+\epsilon_1a_1+...+\epsilon_k)\, |\, \epsilon_j\in\{ -1,1\}$ for $j=1,...,k\}\, .$
So $Q_k(x)$ can be written in the additive form
\begin{equation} Q_k(x)=2^{-k} \sum_{i=1}^{2^k} \cos(2\pi x\, p_i(\theta)) \label{qk}
\end{equation}   where the $p_i$ are the $2^k$ polynomials of the form
$p_i(\theta)=1+\sum_{j=1}^k \epsilon_j\theta^j $ with $\epsilon_j\in\{ -1, 1\} \ .$
Using the notation of Section \ref{s1}, the  $p_i$ belong to $R(\theta ,1).$ 

\begin{Theorem}
If  $R(\theta,1)$ is a coherent set of frequencies, then $\theta$ is Pisot. \label{pi}
\end{Theorem}

\noindent {\it Proof. } We argue by contradiction.  If $\theta$ is not a Pisot number the products $Q_k(x)$ in their additive form \eqref{qk} converge to $0$ uniformly on any compact interval not containing $0.$  The frequencies of $Q_k$ belong to $R(\theta, 1),$ and we have $Q_k(0)=1.$ The trigonometric sums $P_k(x)=Q_k(x+k)$ have the same frequencies which do not satisfy \eqref{e3}. So $R(\theta, 1)$ is not a coherent set of frequencies.
$\Box$ \smallskip 

This theorem is now applied to derive an assertion which is weaker than Feng's theorem and stronger than a theorem of Bugeaud \cite{Bug} from 1996 which assumed that $S(\theta, m)$ is uniformly discrete for all $m\in\NN .$

\begin{Theorem}
If $S(\theta, 4m)$ is uniformly discrete for some $m>\theta-1$ then $\theta$ is a Pisot number.  \label{Bu}
\end{Theorem}

\noindent {\it Proof. } We first show that $R(\theta, m)$ is a Meyer set. \quad
A set $M$ is uniformly discrete if and only if $0$ is an isolated element in $M-M.$ If $S(\theta, 4m)$ is uniformly discrete, then $0$ is isolated in $R(\theta, 4m).$  Now $R(\theta, 4m)=R(\theta, 2m)-R(\theta, 2m).$ Therefore $R(\theta, 2m)$ and its subset $R(\theta, m)$ are uniformly discrete. Since $R(\theta, m)$ is relatively dense when $m>\theta-1,$ it is a Delone set. Its difference set $R(\theta, 2m)$ is uniformly discrete, so $R(\theta, m)$ is a Meyer set.

A Meyer set is a coherent set of frequencies \cite{Meyer72,Mo97}.
Subsets of a coherent set of frequencies are coherent sets, and $R(\theta, 1)\subset R(\theta, m).$ Theorem \ref{pi} now implies that $\theta$ is Pisot.  \hfill $\Box$ \medskip

\section{Self-similar patterns with Pisot factor are uniformly discrete} \label{s3}
Now we consider the self-similarity equation $\Lambda= \theta\Lambda +F$ for a Pisot number $\theta .$ So far $F$ was a set of integers. Rational numbers could be multiplied by their common denominator.  A more general condition is that $F$ consists of algebraic integers in the field $K[\theta]$ generated by $\theta.$ But that would still imply $F\subset\RR .$  In order to include two-dimensional examples, like pentagonal quasicrystallic patterns, we allow for algebraic integers in a larger field $K\supset K[\theta] .$ We assume that $\theta$ is a unit (the constant coefficient of its polynomial is $\pm 1$) and fulfils the following condition, cf. \cite[Section 2]{BaMe}.
\begin{equation}
\mbox{The Galois conjugates $\theta'$ of $\theta$ in $K$ either coincide with $\theta$ or have modulus less than one. }
\label{realcond}\end{equation}
A typical example is the golden ratio $\theta\approx 1.618$  and the fifth cyclotomic field $K.$ 

\begin{Theorem}\label{t3}
Let $\Lambda$ be a closed discrete subset of the ring of integers $R$ in an algebraic field $K$ such that $\Lambda= \theta\Lambda +F$  where $\theta\in K$ is a Pisot unit which fulfils \eqref{realcond}. Then $\Lambda$ and $\Lambda -\Lambda$ are uniformly discrete.  In particular, if $\Lambda$ is a Delone set then it is a Meyer set.
\end{Theorem}

\noindent {\it Proof. } We use the results of \cite{BaMe}: there is a maximal closed and discrete solution $\Lambda^*$ of the equation \eqref{selF} within the ring $R$ of algebraic integers of $K.$ If we prove uniform discreteness for $\Lambda^*,$ it holds for all solutions $\Lambda .$  According to \cite{BaMe}, $\Lambda^*$ is obtained by applying all the mappings  $g_f(z)=\theta z +f$ with $f\in F$ recursively, starting with a certain set $Z_0\subset R$ of initial values. Since $\theta$ is Pisot, the set $Z_0$ is finite \cite[Proposition 8]{BaMe}. As a subset of the complex plane, $Z_0$ is contained in the ball  
\begin{equation} B=\{ z\, | \ (\theta -1)|z|<M\} \quad\mbox{ where } \
M=\max_{f\in F} |f| \ .\label{bal}\end{equation}
Moreover, any image $g_f(z)$ of some $z\in Z_0$ which belongs to $B$ is also in $Z_0.$ For $w\not\in B$ we have the estimate
\[ |g_f(w)|=|\theta w+f|\ge \theta |w|-|f| \ge \theta |w|-M \ge |w| \]
which says that application of the mappings  $g_f$ to points outside $B$ will never lead back to $B.$ Thus $\Lambda^*\cap B= Z_0,$ and 0 is an isolated point of $\Lambda^* .$

The set $\Gamma=\Lambda-\Lambda$ fulfils the equation $\Gamma=\theta\Gamma+G$ with $G=F-F.$ All the previous assumptions of the theorem are valid, so the same proof applies to this equation and shows that 0 is an isolated point of $\Gamma^*=\Lambda^*-\Lambda^*.$ Thus $\Lambda^*$ is uniformly discrete. Now we can go one step further and apply the proof to the equation for difference sets $\Gamma-\Gamma,$ showing that 0 is isolated in $\Gamma^*-\Gamma^*.$  This shows that $\Gamma^*$ is uniformly discrete. As a consequence, for any solution $\Lambda$ of \eqref{selF}, both $\Lambda$ and $\Gamma=\Lambda-\Lambda$ are uniformly discrete. If in addition $\Lambda$ is relatively dense, it must be a Meyer set by the criterion of Lagarias.
\hfill $\Box$ \medskip 

For the case $\theta <2,$ we obtain with Theorem \ref{t1} the following corollary.

\begin{Theorem}
Let $\Lambda\subset \mathbb{R}$ be a Delone set and $q\in (1, 2)$ a real number with
$\Lambda\supseteq q\Lambda+\{0, 1\}. $ \
Then $q$ is a Pisot unit and $\Lambda$ is a Meyer set.
\end{Theorem}

\section{One-dimensional examples}
The simplest self-similarity equation is
\begin{equation} \Lambda = \theta\Lambda +\{ 0,1\} =\theta\Lambda \cup (\theta\Lambda +1)\ .
\label{sel2}\end{equation}
For $\theta =2,$ there are lots of solutions $\Lambda$: the non-negative integers, the negative integers, all integers, all rational numbers, all rational numbers, all rational numbers of the form $n/p$ for a fixed odd number $p,$ and so on. Since we are interested in closed discrete sets $\Lambda ,$ we have to restrict the domain of our solutions.  If we require that $\Lambda\subseteq\ZZ$ then we obtain only the first three solutions. Moreover, if we require $\Lambda$ to be a Delone set, the solution $\Lambda=\ZZ$ is unique. \smallskip

Now let $\theta <2$ be an arbitrary Pisot number.  We restrict our attention to solutions in the ring $R$ of integers of the field $K$ generated by $\theta .$ Since we are on the line, this is a real number field and it makes no sense to consider larger fields.
Clearly, $S(\theta, 1)$ defined in \eqref{Sm} is a solution. However, this is not a Delone set since it contains only positive numbers. If $1/(\theta -1)$ belongs to $R,$ there is a another solution within the negative numbers.  All other solutions are Delone sets. 

\begin{Proposition}\label{p1}
For any given Pisot number $\theta<2,$ all closed discrete solutions $\Lambda\subseteq R$ of \eqref{sel2}, with at least one and at most two exceptions, are Delone sets. 
\end{Proposition} 

\noindent {\it Proof. } We use the notation of \cite{BaMe}, writing $\Lambda= g_0(\Lambda)\cup g_1(\Lambda)$ for \eqref{sel2} where $g_0(x)=\theta x$ and $g_1(x)=\theta x+1$ are expansive similarity maps. Each solution $\Lambda$ is obtained by applying $g_0,g_1$ recursively to the points of a periodic orbit of $\{ g_0, g_1\},$ cf. \cite[Section 4]{BaMe}. If this periodic orbit is the fixed point 0 of $g_0,$ application of $g_0,g_1$ leads to larger and larger positive numbers, and we obtain $\Lambda=S(\theta, 1).$ If we start with the fixed point  $-1/(\theta -1),$ we get negative numbers which grow in modulus. This case makes sense only if $1/(\beta-1)$ is in $R.$  All other periodic orbits are in the interval $[-1/(\theta -1),0]$ where $g_1$ moves points to the right and $g_0$ moves them to the left. The set $\Lambda$ obtained from such an orbit
contains a negative number $-a$ and a positive number $b.$ Let $H=\{ -a, b\}$ and $T=[-a,b].$ Then
\[ \theta T \subseteq 2T = T+H \ , \ \mbox{ and, by multiplication and substitution,}\]
\[ \theta^2 T \subseteq \theta T +\theta H \subseteq T+H+\theta H , \ \mbox{ and by induction} \]
\[ \theta^k T \subseteq T+(H+\theta H +...+\theta^{k-1}H)\subset T+\Lambda \quad\mbox{ for all $k.$}\]
Since the sets $\theta^kT$ fill the line for $k\to\infty ,$ this implies $T+\Lambda =\RR^2.$ Thus $\Lambda$ is relatively dense. By Theorem \ref{t3}, $\Lambda$ is uniformly discrete, hence a Delone and a Meyer set.
\hfill $\Box$ \medskip 

There is a maximal solution $\Lambda^*$  of \eqref{sel2} in $R$ which is a cut-and-project set \cite[Theorem 1]{BaMe}. This fact together with the above proof gives the following corollary.

\begin{Proposition}\label{p2}
For each Pisot number $\theta<2,$ there is a solution $\Lambda^*$ of \eqref{sel2} in the corresponding ring of algebraic integers which is a Meyer set and a cut-and-project set. 
\end{Proposition} 

There seem to be only few solutions of \eqref{sel2} beside the non-Delone minimal solutions $\Lambda_0, \Lambda_1$  and the maximal solution $\Lambda^*.$  

\begin{Example}\label{ex1}
Let $\theta=\tau=\frac{1+\sqrt{5}}{2}.$ Starting from 0, we obtain 
\[ \Lambda_0=S(\tau ,1)=\ZZ_\tau = \{\ \sum_0^K \epsilon_k\tau^k\, | \,  K\in\NN , \epsilon_k\in \{0, 1\} \ \}\ . \]
Usually, this set is extended to the whole line by adding $-\ZZ_\tau $ which is natural from the viewpoint of a numeration system. However, this is neither consistent with self-similarity nor with the cut-and-project method. The fixed point of $g_1(x)=\tau x +1$ is $-\tau ,$ and the corresponding solution of \eqref{sel2} is 
\[\Lambda_1=-\Lambda_0 -\tau .\]
Now $\Lambda_0\cup\Lambda_1$ is a Delone set solution of \eqref{sel2}. However, we should also check the periodic points of $\{ g_0, g_1\}$ within $R=\ZZ[\tau]=\{ m\tau +n\, |\, m,n\in\ZZ\} .$ It turns out that beside the fixed points of $g_0$ and $g_1$ only one cycle of period two exists: $y=-1$ and $z=1-\tau,$ with $g_1(y)=z$ and $g_0(z)=y.$ Starting from this cycle we obtain the maximal solution of \eqref{sel2} in $R:$
\[ \Lambda^*=\{ -1, 1-\tau\}\cup\Lambda_0\cup\Lambda_1\ .\]
Thus we have exactly two Delone sets which solve \eqref{sel2}, and they differ only by two points. 
\end{Example}

So far we considered the equation $\Lambda=\theta\Lambda +F$ only for $F=\{ 0,1\} .$ Proposition \ref{p1} together with Proposition \ref{p3} below shows that for statement on uniform discreteness, Delone and Meyer sets this is really the most general case. We noted already in Section \ref{s1} that the first point of $F$ can always be taken as 0.   However, the number and structure of solutions will change when we consider $F=\{ 0,2\}$ (or $\{0, 1+\tau^{-3}\}$, for instance). Since choosing $F=\{ 0,k\}$ is equivalent to extending the domain of solutions for $F=\{ 0,1\}$ to the set $\frac{1}{k}\cdot R,$ the number of solutions will grow.

\begin{Example}\label{ex2}
For $\theta=\tau,$ we consider the solutions of  $\Lambda = \theta\Lambda +\{ 0,2\} .$ Since $\tau^2=\tau+1,$ the mapping $g_0(x)=\tau x$ acts on $R=\ZZ[\tau]$ as the linear mapping $g(m\tau +n)=(n+m)\tau +m.$ If $m$ and $n$ are even, both components of $g(m\tau +n)$ are even. If one of $m,n$ is odd, this also holds for the image under $g.$ So we have two disjoint invariant spaces of $g_0$ which are also invariant under $g_2(x)=\tau x+2$ which now replaces $g_1.$ As a consequence, we have all solutions of Example \ref{ex1}, multiplied by 2, as subsets of the even subspace. On the odd subspace, we look for periodic points of $\{ g_0,g_2\}$ in the interval $[-2\tau ,0]$ where $-2\tau$ is the fixed point of $g_2.$
There are two cycles of period 3: $-\tau, -\tau +1, -1$ and $-\tau, -\tau -1, -2\tau +1$ which have the common point $-\tau .$
This irreducible periodic network generates only one solution. Including unions of solutions for odd and even numbers we obtain seven Delone solutions.
\end{Example}

\section{Two-dimensional examples}
For the golden ratio $\theta=\tau, $ a number of two-dimensional Meyer sets $\Lambda^*$ with decagonal symmetry were constructed and discussed in \cite{BaMe}.  We conclude our paper by briefly introducing new examples with other Pisot factors and with three pieces.
The equation is
\begin{equation} \Lambda = \theta\Lambda + F \quad \mbox{ with } F=\{ a,b,c\} \ .
\label{sel3}\end{equation}
Proposition \ref{p4} below says that for the smallest four Pisot numbers, relatively dense solutions can exist. In Proposition \ref{p4a}, we prove for the plastic number $\beta\approx 1.3247$ and the fourth Pisot number $\theta\approx 1.4656$ the existence of Meyer sets $\Lambda$ which satisfy \eqref{sel3}.

Since the case of collinear digits was studied in \cite{Feng16} and Section \ref{s1}, we now assume that $a,b,c$ are not on a line.  While above we have chosen the origin so that $0$ belongs to $F,$ it now seems reasonable to chose the coordinates to that $a+b+c=0.$
The following assertion is easy to prove and shows that all choices of $a,b,c$ are equivalent, as long as general properties of the solutions are studied. Note that any two triples $a,b,c$ of non-collinear points in the plane with $a+b+c=0$ can be transformed into each other by a linear map $h(x)=Ax.$

\begin{Proposition}\label{p3}
Let $\Lambda$ be a closed discrete subset in $\RR^n$ and $A$ a non-singular $n\times n$ matrix.
\begin{enumerate}
\item[(i)] $\Lambda$ is uniformly discrete, relatively dense, a Delone set, or a Meyer set if and only if $A(\Lambda)$ has the respective property.
\item[(ii)] For a given $q>1$ and a finite set $F\subset\RR^n,$ the set $\Lambda$ satisfies $\Lambda = q\Lambda +F$ if and only if $A(\Lambda )=qA(\Lambda) +A(F).$  \vspace{-2ex}
\end{enumerate} \hfill $\Box$ \vspace{-1ex}
\end{Proposition} 

Thus without loss of generality we can assume that the three-point set $F$ consists of the third roots of unity: $F=\{ 1,\omega, \omega^2\}$ with $\omega=-\frac12 +\frac12\sqrt{3}i.$ These points generate the hexagonal lattice which consists of the integers of the third cyclotomic field. This is a natural symmetric choice. A corresponding ring $R$ of algebraic integers should contain $\theta$ and all related algebraic integers which would require calculations in the field $K=K[\omega,\theta]$ generated by $\omega$ and $\theta .$ Here we focus on more elementary arguments.

\begin{Proposition}\label{p4}
 For $\theta<3/2,$ every solution $\Lambda$ of \eqref{sel3} which contains 0 is relatively dense.   
\end{Proposition} 

\noindent {\it Proof. } Let $T$ be the convex hull of $2F.$ We then have 
$T+F=(3/2)T.$  Since $\theta<3/2$ it implies 
\begin{equation}\theta T\subset T + F\ .\label{TH} \end{equation} 
 By iteration of \eqref{TH} we obtain for every $m>1$
\begin{equation}  
\theta^m  T\subset  T + \{ F+ \theta\, F+\cdots+ \theta^{m-1}\, F\}  \label{THm} \end{equation}
The union of the dilated triangles $\theta^m T$ is the complex plane. On the other hand, each solution $\Lambda$
of $\Lambda=\theta\Lambda+F$ with $0\in\Lambda$ must contain the bracketed term on the right-hand side of \eqref{THm}, for every $m.$ Thus $T+\Lambda$ must be the whole plane which ends the proof. \hfill $\Box$
\smallskip

There are four Pisot numbers which are smaller than $3/2,$ as shown by Dufresnoy and Pisot \cite{DP55} in the 1950s. The smallest Pisot number $\beta\approx 1.3247,$ called the plastic number, has the minimal polynomial $\beta^3-\beta-1.$ The
fourth Pisot number $\theta\approx 1.4656$ also has a polynomial fo degree 3, $\theta^3-\theta^2-1.$

\begin{Proposition}\label{p4a}
For the first and fourth Pisot number, consider the equation \eqref{sel3} where $a,b$ are algebraic integers in the field $K[\theta]$ and $c=-a-b.$
The recursive application of $g_a(z)=\theta z+a, \ g_b(z)=\theta z+b,$ and $g_c(z)=\theta z+c$ with initial point $0$ produces a solution $\Lambda$  which is a Meyer set. It can be written as
 \[  \Lambda=\bigcup_{m=0}^\infty \Lambda_m\quad\mbox{ with } \quad \Lambda_0=\{0\},\ \Lambda_m= F+\theta F +...+\theta^m F \mbox{ for } m>0 .\]
\end{Proposition} 

\noindent {\it Proof. } In view of Theorem \ref{t3} and Proposition \ref{p4}, we need only show that $0$ is a periodic point of the iterated function system $\{ g_a,g_b,g_c\} .$ Then all points in the recursion will repeat periodically, which implies the  formula for $\Lambda .$ The periodic representations
\[ g_c^8g_ag_cg_b^2g_a(0)=0  \quad \mbox{ for the plastic number and } \]
\[ g_c^4g_bg_cg_bg_a(0)=0  \quad \mbox{ for the fourth Pisot number } \]
were found with the algorithm of \cite{BaMe} by computer. We show how the second one can be found by hand. Writing $(x_1,x_2,x_3)$ for $x_1\theta^2+x_2\theta+x_3$, and using $\theta^3=\theta^2+1,$  the map $g_v$ with $v\in\{ a,b,c\}$ has the form 
\[ g_v(x_1,x_2,x_3)=(x_1+x_2,x_3,x_1+v). \quad\mbox{ Thus }\quad g_v^{-1}(y_1,y_2,y_3)=(y_3-v, y_1-y_3+v, y_2). \]
Using the relation $a+b+c=0$ we obtain  $g_c^{-1}(0,0,0)=(-c,c,0), \ g_c^{-1}(-c,c,0)=(-c,0,c),$ and $g_c^{-1}(-c,0,c)=(0,-c,0).$  While the choice $v=c$ in the first step was arbitrary, it was necessary in the other two steps in order to keep the result simple. For the next two steps there are again different choices. $g_c^{-1}(0,-c,0)=(-c,c,-c), \ g_b^{-1}(-c,c,-c)=(a,b,c), \ g_c^{-1}(a,b,c)=(0,a,b).$ Now $g_b^{-1}$ and  $g_a^{-1}$ lead back to $(0,0,0).$ In this case there are eleven other periodic representations of period 8. Apparently there are no shorter ones, however. \hfill $\Box$\smallskip

\begin{figure}[h!t]
\begin{center}
\includegraphics[width=0.495\textwidth]{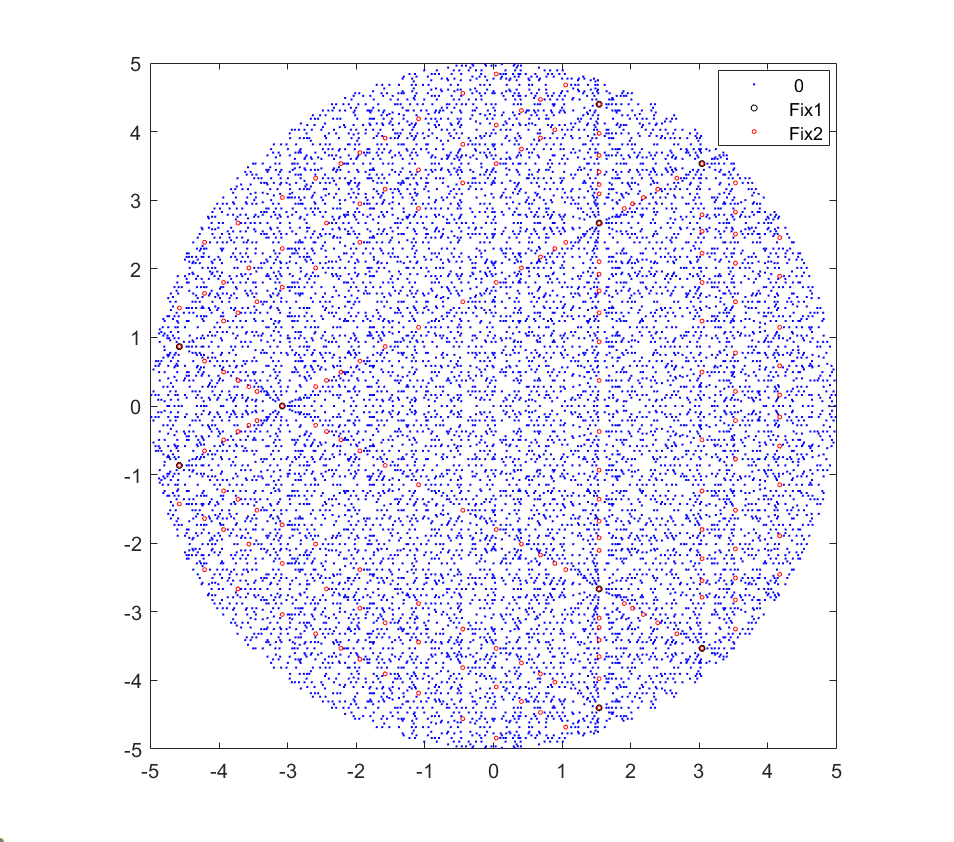} \   \includegraphics[width=0.47\textwidth]{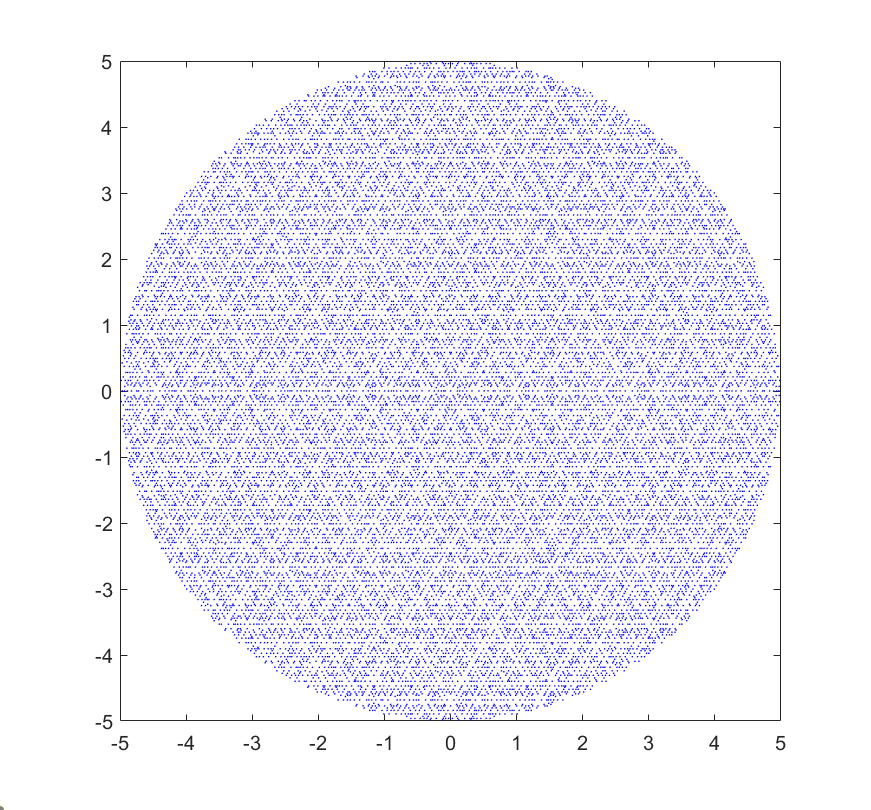} 
\end{center}
\caption{Patterns with three-fold symmetry for the plastic number $\beta\approx 1.3247$ restricted to the disk $|z|<5.$ Left: Solutions of \eqref{sel3} based on the initial point 0 (11000 blue points), the fixed points of the maps $g_k$ ($3\times 3$ black circles), and all fixed points of the mappings $g_kg_j$ (183 red circles). \  Right: The construction of Proposition \ref{p5} with $H=\{ 0,a,b,c\}$ and initial point 0 shows hexagonal symmetry (23000 points). }\label{fig1}      
\end{figure}

For the plastic number $\beta,$ the set $\Lambda$ obtained from the initial point 0 is shown on the left of Figure \ref{fig1}. There are around 11000 points with modulus smaller than 5. For the fourth Pisot number, this circle contained only 1159 points. Simulations for the second and third Pisot number gave extremely dense patterns, but no periodic representation of 0 was found. 

As in Example \ref{ex1}, we did also try to use the fixed points of $g_a,g_b,g_c$ as initial points for the recursion of the mappings. They generate triangular patterns in the plane which are not relatively dense. For $a=1,$ the fixed point is $-1/(\beta-1)\approx -3.07,$ and only three points of the generated set marked by black circles are seen in our circular window. When we start the recursion with the fixed points of the mappings $g_kg_j,$ we obtain the 183 points marked in red.  Actually these fixed points are rationals with denominator $\beta^2-1.$  The minimal polynomial for $\beta$ is $\beta^3-\beta-1,$ however. Thus $\beta(\beta^2-1)=1,$ and $1/(\beta^2-1)$ is an algebraic integer in this case.  

There seem to be many other solutions of $\Lambda= g_a(\Lambda)\cup g_b(\Lambda)\cup g_c(\Lambda)$ which can be studied by calculations in the spirit of \cite{BaMe}.  The solution $\Lambda$ with starting point 0 can be considered as a two-dimensional analogon of $S(\theta,1)$ in \eqref{Sm}.  We can also construct a set which corresponds to the difference set $R(\theta,1)$ in \eqref{Rm}.

\begin{Proposition}\label{p5}
 For each Pisot number $\theta<3/2$ and non-collinear points $a,b,c$ in the plane with $a+b+c=0,$ there is a solution $\Lambda$ of the equation $\Lambda= \theta\Lambda +H$ with $H=\{ 0,a,b,c\}$ which is a Meyer set. It can be written as 
 \[  \Lambda=\bigcup_{m=1}^\infty \Lambda_m\quad\mbox{ with } \quad \Lambda_m= H+\theta H +...+\theta^m H \ .\]
\end{Proposition} 

\noindent {\it Proof. } Proposition \ref{p4} says that $\Lambda$ is relatively dense. Let $R=R(\theta,1)$ defined in \eqref{Rm}. Every $\lambda\in\Lambda$ can be written as a finite sum
\[ \lambda = \sum \epsilon_k \theta^k \eta_k \]
where $\epsilon_k\in\{ 0,1\}$ and $\eta_k\in\{ a,b,c\} .$ Since $c=-(a+b),$ each $\lambda$ belongs to $aR+bR .$ Now $R$ is known to be a Meyer set, cf. \cite{Feng16} and Proposition \ref{p2}, and $a,b$ are linearly independent vectors. So $aR+bR$ and $(aR+bR)-(aR+bR)=a(R-R)+b(R-R)$ are uniformly discrete. Moreover, $R-R\subseteq R+G$ for a finite set $G$ implies $a(R-R)+b(R-R)\subseteq (aR+bR)+(aG+bG)$ where $aG+bG$ is also finite. Thus $aR+bR$ and its subset $\Lambda$ are Meyer sets. \hfill $\Box$\medskip

The result for the plastic number $\beta$ can be seen in Figure \ref{fig1} on the right, with a very clear hexagonal symmetry. The argument above, with $4/3$ instead of $3/2,$ can be used to derive the following three-dimensional example.  The plastic number is the only Pisot number below $4/3.$ 

\begin{Proposition}\label{p5}
 For the plastic number $\beta$ and points $a,b,c,d$ in $\RR^3$ which are not on a common plane and fulfil $a+b+c+d=0,$ there is a solution $\Lambda$ of the equation $\Lambda= \beta\Lambda +H$ with $H=\{ 0,a,b,c,d\}$ which is a Meyer set. It can be written as 
 \[  \Lambda=\bigcup_{m=1}^\infty \Lambda_m\quad\mbox{ with } \quad \Lambda_m= H+\beta H +...+\beta^m H \ .\]
\end{Proposition} \medskip

{\bf Acknowledgement. }  Yves Meyer gratefully acknowledges stimulating discussions with Yann Bugeaud, involving in particular the hint to Feng's theorem.

\medskip\noindent
Institute of Mathematics, University of Greifswald, Germany, \url{bandt@uni-greifswald.de}\vspace{1ex}\\
Acad\'emie des Sciences, 23 quai de Conti, 75006 Paris, France, \url{yves.meyer305@orange.fr}
 
\end{document}